\documentclass[11pt]{amsart}
\usepackage[utf8]{inputenc}
\usepackage{amsmath}
\usepackage{amsthm}
\usepackage{fullpage,url,amssymb,enumitem,colonequals, mathrsfs,comment}
\usepackage{microtype}
\usepackage{hyperref}
\usepackage{xcolor}
\usepackage{tikz}
\usepackage{tikz-cd}
\usepackage[pdftex]{}
\usepackage[noabbrev,capitalize, nameinlink]{cleveref}

\newcommand{\PP}{\mathbb{P}}

\newcommand{\Q}{\mathbb{Q}}
\newcommand{\R}{\mathbb{R}}
\newcommand{\Z}{\mathbb{Z}}

\newcommand{\OO}{\mathscr{O}}

\DeclareMathOperator{\divv}{div}
\DeclareMathOperator{\Div}{Div}

\DeclareMathOperator{\ord}{ord}
\DeclareMathOperator{\Pic}{Pic}

\DeclareMathOperator{\supp}{supp}

\DeclareMathOperator{\Bl}{Bl}

\newcommand{\wequiv}{\equiv_{\text{W}}}
\newcommand{\WDiv}{\text{WDiv}}
\newcommand{\cequiv}{\equiv_{\text{C}}}
\newcommand{\CDiv}{\text{CDiv}}
\newcommand{\cyc}{\text{cyc}}

\newtheorem{theorem}{Theorem}[section]
\newtheorem{lemma}[theorem]{Lemma}
\newtheorem{corollary}[theorem]{Corollary}
\newtheorem{proposition}[theorem]{Proposition}

\theoremstyle{definition}
\newtheorem{definition}[theorem]{Definition}

\newtheorem{example}[theorem]{Example}

\theoremstyle{remark}
\newtheorem{remark}[theorem]{Remark}

\usepackage{microtype}
\title{Interpolation of low degree points on curves}
	
\author{Eden Granot}
\email{eden.granot@mail.huji.com}
\address{Eden Granot, Einstein Institute of Mathematics\\
	The Hebrew University of Jerusalem, Jerusalem, Israel}
\date{\today}

\begin{document}
	\begin{abstract}
		We study points of moderately low degree on a curve $C$ over a number field, which is embedded on a nice toric surface $S$. Recently, Smith and Vogt related the linear equivalence classes of such points to intersections of $C$ with curves in the ambient surface. We show that when $C$ is sufficiently effective and ample with simple singularities, all but finitely many sufficiently low degree points are obtained via intersections of $C$ with curves in the ambient surface $S$. Our methods make use of the toric geometry of the ambient surface $S$ to produce curves which interpolate the low degree points on $C$. Furthermore, we generalise a result by Debarre and Klassen to singular plane curves and higher degrees of points. 
	\end{abstract}
	\maketitle
	
	\section{Introduction}
	Let $k$ be a number field. Let $C$ be a projective geometrically integral curve over $k$. In this paper, we are interested in understanding the set of degree $e$ points on $C$,
	\begin{equation*}
		|C|_e \colonequals\{P\in |C|\mid \deg(P)=e\},
	\end{equation*}
	where $|C|$ is the set of closed points on $C$ and the degree $\deg(P)$ is the degree of the residue field extension $[k(P):k]$. 
	
	When $e=1$, this is the set of $k$-rational points $C(k)$. When $C$ is smooth and the genus of $C$ is at least $2$, by Faltings' theorem \cite{a07a8b9c-f9ee-30b2-8d26-66e0414778da}, the set $|C|_1$ is finite. Surprisingly, for higher degree points on general curves, even for small values of $e$ relative to invariants of $C$, there is no such analogue of finiteness. However, for some classes of curves and respective degrees, we know descriptions of $|C|_e$. 
	
	Suppose that $C \subseteq \PP^2_k$ is a smooth plane curve of degree $d$. When $C$ admits rational points there is a natural source of degree $d-1$ points on $C$: intersections of $C$ with a line $\ell \subseteq \PP^2_k$ passing through a rational point of $C$. Debarre and Klassen \cite{Klassen1992PointsOL} showed that for $d\geq 8$, all but finitely many degree $e\leq d-1$ arise as such intersections. Their result is proved using tools from a paper of Coppens and Kato \cite{Coppens1991TheGO} which generalised Max Noether's gonality theorem to certain singular plane curves. The result of Debarre--Klassen can be seen as an arithmetic counterpart to Max Noether's gonality theorem.
	
	Later on, Smith and Vogt \cite{10.1093/imrn/rnaa137} generalised this result to curves on surfaces other than $\PP_k^2$, using ideas from Reider's method \cite{0a1ebf56-1d8f-3732-ba76-902ace340afa}. The same ideas were used in Lazarsfeld’s calculation of the minimal gonality of complete intersection curves \cite{lazarsfeld1997lectures}. Among their results, they showed that when $C\subseteq S$ is an ample nice curve inside a nice surface of irregularity $h^1(S,\OO_S)=0$, there are only finitely many linear equivalence classes of degree $e<\frac{C^2}{9}$ points. Furthermore, the linear equivalence class of such a point which moves can be described via intersections in the ambient surface $S$. 
	
	There is a crucial difference between the results of Debarre--Klassen and Smith--Vogt: while Debarre--Klassen describes all but finitely many degree $e$ points, Smith--Vogt only gives information on linear equivalence classes. The principal aim of this paper is to find a generalisation of Debarre--Klassen to possibly singular curves on toric surfaces that goes beyond linear equivalence classes. We prove the following theorem (\cref{maintoric-proof}).
	
	\begin{theorem}
		\label{maintoric}
		Let $S$ be a nice toric surface. Then there exists an explicit constant $\lambda(S)\in (-\infty,2]$ depending only on $S$, such that the following holds. Let $C\subseteq S$ be a geometrically integral closed curve with simple singularities of multiplicities $\delta_1,\dots, \delta_n$ such that $C+K_S>0$. Assume furthermore that the normalisation of $C$ is ample on the blowup of $S$ at the singular points of $C$. Then there exists a closed curve $D\subseteq S$ with $C\cdot D\leq \frac{C^2}{2}$, and such that for all positive integers
		\begin{equation*}
			e<\min \left(\frac{C^2-\sum_{i=1}^{n}\delta_i^2}{9}, \frac{C^2}{4}+\lambda(S)\right),
		\end{equation*}
		the set $|C|_e$ of degree $e$ points on $C$ admits the following decomposition
		\begin{equation*}
			|C|_e=F\cup \bigcup_{i=1}^m S_i,
		\end{equation*} 
		where $F$ is a finite set, and for each $i$, the set $S_i$ is the set of degree $e$ points obtained by intersecting $C$ with the linear system of curves linearly equivalent to $D$ that pass through a certain effective Weil divisor $B_i\subseteq C$, that is, 
		\begin{equation*}
			S_i = \left\{P\in|C|_e \text{ }\bigg|\text{ } \exists D'\in |D|:P=C\cap D'-B_i\right\}.
		\end{equation*}
		Furthermore, the finite set $F$ is the set of all singular points and regular points of degree $e$ which do not move as Cartier divisors.
	\end{theorem}
	
	\begin{remark}
		\label[remark]{blample}
		The question of whether $\widetilde{C}$ is ample on the blowup $\Bl(S)$ of $S$ along the singular points of $C$ is closely related to the Seshadri constants $\epsilon(S,\OO_S(C); x)$. On toric varieties, there are various results regarding the values of the Seshadri constants; see for example \cite{Rocco1997GenerationOK} and \cite{Ito+2014+151+174}. In \cref{ample} we show how to use these results to produce explicit sufficient conditions for the ampleness of $\widetilde{C}$ on $\Bl(S)$.
	\end{remark}
	
	\begin{remark}
		\label[remark]{basebig}
		In the case when $S=\PP_k^2$ is the projective plane, the constant $\lambda(S)$ is $-\frac{1}{4}$; see \cref{lambdap2}. Let $C\subseteq S$ be a nice curve of degree $d\geq 4$, then the curve $D\subseteq S$ from the above theorem will be of degree $m\colonequals \lfloor \frac{d}{2}\rfloor - 1$. Suppose $e<\frac{d^2}{9}$, then all but finitely many degree $e$ points on $C$ are obtained via intersections of $C$ with a degree $m$ curve passing through certain divisors $B_i\subseteq C$. The degree of $B_i$ is
		\begin{equation*}
			\deg(B_i) = md-e = d\left\lfloor \frac{d}{2}\right\rfloor -d -e > \frac{d(d-1)}{2}-d-\frac{d^2}{9},
		\end{equation*}
		which is way bigger then $\frac{d^2}{9}$, and in particular bigger then $e$. Therefore, the divisors $B_i$ may be hard to understand; we would like a version of the above theorem where the degrees $\deg(B_i)$ are minimised in analogy to Debarre--Klassen. It is worth noting that by throwing components of $D$, one can obtain a curve $F\subseteq S$ of degree $m_0\leq m$ which induces a decomposition on the degree $e$ points of $C$; in this new decomposition, the base locus divisors $B'_i$ have degrees satisfying $\deg B'_i\leq m_0^2\leq m^2$. The existence of such "small" $F$ which induce a decomposition is mainly due to the vanishing of $H^1(S,-)$ on line bundles.
		
		Later on, at \cref{mainexample} we show that $\deg(B_i)$ cannot be completely minimised on surfaces other than $\PP_k^2$, even on the Hirzebruch surface $F_1$. 
	\end{remark}
	For singular curves on $\PP_k^2$, we prove the following result (\cref{plane-proof}) which generalises the theorem of Debarre--Klassen to singular curves and degrees higher than $d-1$.
	
	\begin{theorem}
		\label{plane}
		Let $C\subseteq\mathbb{P}_{k}^{2}$ be a geometrically integral plane curve of degree $d\geq 4$ with $\delta \leq \frac{d-3}{3}$ ordinary nodes and cusps as its singularities. Let
		\begin{equation*}
			e< \max\left(\frac{d^2-4\delta}{9}, \text{   } \frac{1}{2}\left(\left\lceil \frac{d+\sqrt{d^2-36\delta}}{6}\right\rceil\left(d-\left\lceil \frac{d+\sqrt{d^2-36\delta}}{6}\right\rceil\right)-\delta\right)\right)
		\end{equation*}
		be a positive integer. Then the set of degree $e$ points on $C$ admits a decomposition
		\begin{equation*}
			|C|_e=F\cup\bigcup_{i=1}^{n}S_i,
		\end{equation*}
		where $F$ is a finite set, and for each $i$, the set $S_i$ is the set of degree $e$ points obtained by intersecting $C$ with the linear system of degree $m$ curves passing through a certain effective Weil divisor $B_i\subseteq C$. The finite set $F$ is the set of all singular points and regular points which do not move as Cartier divisors. In case that $S_i$ are nonempty, the degree $m$ is a positive integer such that
		\begin{equation*}
			m(d-m)\leq e+\delta < (m+1)(d-(m+1)),
		\end{equation*}
		and in particular the divisors $B_i$ are of degree $\deg B_i = md-e < \frac{e}{2}$.
	\end{theorem}
	
	\begin{remark}
		In our settings, the inequality $\frac{d^2-4\delta}{9} \geq \frac{1}{2}\left( \frac{d+\sqrt{d^2-36\delta}}{6}\left(d- \frac{d+\sqrt{d^2-36\delta}}{6}\right)-\delta\right)$ always holds. Thanks to the ceilings, sometimes (but not always), we get strong inequality in the other direction as well. For example, when $C$ is smooth, $\delta = 0$ and the second term equals
		\begin{equation*}
			\left\lceil\frac{d}{3}\right\rceil\left(d-\left\lceil\frac{d}{3}\right\rceil\right).
		\end{equation*}
		This quantity is always bigger than or equal to the first term $\frac{d^2}{9}$, with equality if and only if $d$ is divisible by $3$.
	\end{remark}
	
	\begin{remark}
		Since $\deg B_i < \frac{e}{2} < e$, the theorem can be used in an inductive way to better understand $B_i$. Indeed, we may decompose $B_i$ into a sum of closed points $P_1+\dots +P_\ell$. Note that $\deg P_j \leq \deg B_i<e$, hence if $P_i$ is a regular point which moves by the above theorem, we may write $P_j=C\cap D'-B_{i,j}$ for some curve $D'\subseteq \PP_k^2$ of degree $m_j$ satisfying 
		\begin{equation*}
			m_j(d-m_j)\leq \deg P_j+\delta < (m_j+1)(d-(m_j+1)),
		\end{equation*}
		and effective Weil divisor $B_{i,j}$ on $C$ with degree  $\deg B_{i,j}<\frac{\deg P_j}{2}<\frac{e}{4}$. We may repeat this process until we are left with only singular points and regular points which do not move as Cartier divisors.  
	\end{remark}
	
	\section{Background, Notations and conventions}
	Throughout this paper, all varieties are assumed to be over a fixed number field $k$ unless stated otherwise. A variety is said to be \emph{nice} if it is smooth, projective and geometrically integral. If $C$ is a nice curve, we denote by $C^{(e)}$ the $e$-th symmetric power of $C$; this space parametrises effective divisors of degree $e$ on $C$. We also denote by $W_eC$ the so-called Brill--Noether locus, which is defined as the image of $C^{(e)}$ under the Abel--Jacobi map. 
	
	Given a variety $X$, we denote by $\CDiv(X)$ the group of Cartier divisors on $X$. Similarly, $\WDiv(X)$ denotes the group of Weil divisors on $X$. Linear equivalence of Cartier divisors is denoted by $\cequiv$, and of Weil divisors is denoted by $\wequiv$. We denote the canonical map $\CDiv(X)\rightarrow \WDiv(X)$ by $D\mapsto \cyc(D)$. When $X$ is smooth, the map $\cyc$ is an isomorphism and we write $\Div(X)\colonequals \WDiv(X)$. For definitions see \cite[Section 3.2]{fulton1984intersection}. Given two Cartier divisors $D,E\in \CDiv(X)$, we write $D\geq E$ when $H^0(X,D-E)$ does not vanish. We write $D>E$ when $D\geq E$ and $D-E$ is not principal. 
	
	Let $C$ be a projective curve. A linear system $\mathfrak{d}$ on $C$ is called a $g^r_e$ if $\dim \mathfrak{d} = r$ and $\deg \mathfrak{d} = e$. We say that a Cartier divisor $D$ on $C$ \emph{moves} or \emph{moves in a pencil} if $h^0(C,D)\geq 2$. We say that $D$ is basepoint-free if the base locus $\bigcap_{D'\in |D|} \supp(D') $ is empty.
	
	The main arithmetic input of this paper comes from a conjecture of Lang, which was proven by Faltings \cite{FALTINGS1994175}. 
	\begin{theorem}
		\label{lang}
		Let $A$ be an abelian variety over a number field $k$ and let $X\subseteq A$ be a closed subvariety. Then there are finitely many abelian translates $X_i=x_i+B_i\subseteq X$ such that each $k$-rational point of $X$ lies on one of the $X_i$.
	\end{theorem}
	
	The linear equivalence classes of degree $e$ points on a nice curve $C$ are parametrised by $W_e C(k)$, which lies on an abelian variety. Hence, the following geometric theorem of Smith and Vogt \cite[Theorem 2.9]{10.1093/imrn/rnaa137} implies that in some situations the set of linear equivalence classes of degree $e$ points is finite.
	\begin{theorem}
		\label[theorem]{smithvogt}
		Let $S$ be a nice surface with irregularity $h^1(S,\OO_S)=0$ and let $C\subseteq S$ be a nice ample curve on $S$. Then for $e<\frac{C^2}{9}$, the locus $W_eC_{\overline{k}}$ contains no positive-dimensional abelian varieties.
	\end{theorem}
	
	As a corollary, we obtain an analogous statement for singular curves, which we will use later.
	\begin{corollary}
		\label[corollary]{singsmithvogt}
		Let $S$ be a nice surface with irregularity $h^1(S,\OO_S)=0$ and let $C\subseteq S$ be a geometrically integral closed curve on $S$ with simple singularities\footnote{A singularity is said to be simple if it is resolved by a single blow-up.}. Assume that the normalisation $\widetilde{C}$ is ample on the blowup $\Bl(S)$ of $S$ along the singular points of $C$. Denote by $\delta_1,\dots, \delta_n$ the multiplicities of the singular points on $C$. Then for 
		\begin{equation*}
			e<\frac{C^2-\sum_{i=1}^{n}\delta_i^2}{9},
		\end{equation*}
		the locus $W_e\widetilde{C}_{\overline{k}}$ contains no positive-dimensional abelian varieties. In particular, the set of rational points $W_e\widetilde{C}(k)$ is finite.
	\end{corollary}
	\begin{proof}
		By \cite[Chapter V Proposition 3.4]{hartshorne1977algebraic}, the irregularity of $\Bl(S)$ is the same as the irregularity of $S$. Note that the self intersection number of $\widetilde{C}$ on $\Bl(S)$ is $\widetilde{C}^2=C^2-\sum_{i=1}^{n}\delta_i^2$. Applying \cref{smithvogt} to $\widetilde{C}$ on $\Bl(S)$, we obtain the desired result.
		
		The finiteness of $W_e\widetilde{C}(k)$ follows from Lang's conjecture (\cref{lang}).
	\end{proof}
	
	\section{Interpolations of degree $e$ points}
	
	In this section, we prove that when a curve in our ambient surface $D\subseteq S$ satisfies some nice properties, it induces a decomposition on $|C|_e$; this will be key to proving our main theorems. The results are mostly clear when $C$ is a nice curve; we carefully treat the singular case as well. Our main goal is to cut out degree $e$ points on $C$ using curves on $S$. We begin with a definition which singles out a minimalist set of algebro-geometric conditions needed to induce the decompositions in our main theorems.
	
	\begin{definition}
		\label[definition]{degree-e-interpolation}
		Let $C\subseteq S$ be a geometrically integral closed curve on a nice surface, and let $e$ be a positive integer. We say that a closed curve $D\subseteq S$ is \emph{an interpolation of degree $e$ points on $C$} if the following properties hold.
		\begin{enumerate}
			\item\label{resisinter} $C\cdot D < C^2$.
			
			\item\label{surjective} The map $H^0(S,D)\rightarrow H^0(C,D|_C)$ is surjective.
			
			\item\label{lift} For all degree $e$ Cartier divisors $E\in \CDiv(C)$ moving in a basepoint-free pencil, the divisor $D|_C-E$ has global sections.
		\end{enumerate}
	\end{definition}
	
	Next, we briefly discuss restrictions of divisors on singular curves; for a more comprehensive discussion, see \cite[Subsection 7.1.3]{liu2002algebraic}. Let $C\subseteq S$ be an integral closed curve in a nice surface, and let $\iota:C\rightarrow S$ be the inclusion. Let $D\subseteq S$ be a closed curve. As $S$ is nice, the curve $D$ defines a unique Cartier divisor on $S$. We can always define a restriction of $D$ to $C$ as a Cartier divisor on $C$ representing the divisor class of the pullback $\iota^*[D]$. This choice is non-canonical, which can cause issues. When the scheme $C\cap D$ is finite, we can define the restriction $D|_C\in \CDiv(C)$ in the following canonical way. Let $\{(U_i, f_i)\}_{i\in I}$ be local equations defining $D$ in $S$ as an effective Cartier divisor. We define
	\begin{equation*}
		D|_C \colonequals \{(U_i\cap C,f_{i}|_C)\}_{i\in I}.
	\end{equation*}
	This is indeed a Cartier divisor since $D$ does not contain $C$. We note that $D|_C$ is effective. One can define an isomorphism of $\OO_C$-modules $\iota^{*}\OO_S(D)\rightarrow \OO_C(D|_C)$ which makes the map of $\OO_S$-modules $\OO_S(D)\rightarrow \iota^{*}\OO_S(D)$ the usual restriction $f\mapsto f|_C$. A direct calculation shows that the associated Weil divisor $\cyc(D|_C)$ is the scheme-theoretic intersection $C\cap D$. As long as the scheme $C\cap D$ is finite, we will use $D|_C$ to denote the canonical definition of the restriction.
	
	With that canonical construction in mind, we prove the following lemma, which will allow us to use properties \eqref{resisinter}, \eqref{surjective} when showing that interpolation of degree $e$ points induces a decomposition of $|C|_e$.
	
	\begin{lemma}
		\label[lemma]{varying-intersection}
		Let $C\subseteq S$ be an integral closed curve in a nice surface. Assume that $C$ is ample on $S$. Let $D\subseteq S$ be a closed curve such that $D\cdot C<C^2$. Then 
		\begin{enumerate}
			\item\label{finite-intersection} For all curves $D'\subseteq S$ linearly equivalent to $D$, the scheme $C\cap D$ is finite. 
			
			\item\label{varying-cyc} Assume furthermore that the map on global sections $H^0(S,D)\rightarrow H^0(C,D|_C)$ is surjective. Then for all effective Cartier divisors $E\in \CDiv(C)$ linearly equivalent to $D|_C$, there exists a curve $D'\subseteq S$ linearly equivalent to $D$ such that $\cyc(E)=C\cap D'$.
		\end{enumerate}
	\end{lemma}
	\begin{proof}
		Let $D'\subseteq S$ be a curve linearly equivalent to $D$. Suppose by way of contradiction that $C$ is an irreducible component of the curve $D'$. As divisors, we may write $D'=C+R$ for some effective divisor $R\in \Div(S)$. Since $C$ is ample, $R\cdot C\geq0$ and hence
		\begin{equation*}
			C^2 > D\cdot C=(C+R)\cdot C \geq C^2
		\end{equation*}
		This is a contradiction, so $C$ is not an irreducible factor of $D'$. Hence, the scheme $C\cap D'$ is finite, which proves \eqref{finite-intersection}.
		
		We now prove \eqref{varying-cyc}. Let $E\in \CDiv(C)$ be an effective Cartier divisor linearly equivalent to $D|_C$. Let $f\in H^0(C,D|_C)$ be such that 
		\begin{equation*}
			E=D|_C+\divv(f).
		\end{equation*}
		Since the map of global sections $H^0(S,D)\rightarrow H^0(C,D|_C)$ is surjective, there exists $g\in H^0(S,D)$ such that $f=g|_C$. Define $D'\colonequals D+\divv(g)$; this is an effective Cartier divisor on $S$ and hence a closed curve. By \cref{finite-intersection} and the above discussion, $D'|_C$ is a Cartier divisor with associated Weil divisor $\cyc(D'|_C)=C\cap D'$. By the definition of the restriction, we see that
		\begin{equation*}
			D'|_C = D|_C + \divv(g|_C) = E.
		\end{equation*} 
		Thus $\cyc(E) = C\cap D'$ as claimed.
	\end{proof}
	
	We also prove the following lemma, asserting that closed points which move are in fact basepoint-free. This will allow us to use property \eqref{lift} in the upcoming proof that interpolation of degree $e$ points induces a decomposition of $|C|_e$.
	
	\begin{lemma}
		\label[lemma]{basepoint-free}
		Let $C$ be a geometrically integral proper curve. Let $P\in C_\text{reg}$ be a regular closed point on $C$. Treated as a Cartier divisor on $C$, assume that the divisor $P$ moves in a pencil. Then $P$ moves in a basepoint-free pencil.
	\end{lemma}
	\begin{proof}
		Let $f\in H^0(C,P)$ be a non-constant global section. In particular $\ord_P(f)\geq -1$. Suppose by way of contradiction that $\ord_P(f)\geq 0$. Since $P$ is a regular point on $C$, we conclude that $f\in \OO_C(C)$ is a regular global section. But $C$ is geometrically integral and proper, so $\OO_C(C)=k$ and $f$ must be constant. This is a contradiction, so $\ord_P(f)=-1$. Hence $\ord_P(P+\divv(f))=0$ and as $P$ is a regular point on $C$ we conclude that 
		\begin{equation*}
			P\notin\supp(P+\divv(f)).
		\end{equation*}
		But the base locus of $P$ is contained in $\supp(P)=\{P\}$ and in $\supp(P+\divv(f))$, so it must be empty. Thus $P$ moves in a basepoint-free pencil.
	\end{proof}
	
	Now we are ready to state and prove a proposition which justifies \cref{degree-e-interpolation}, and will be key to proving our main theorems.
		
	\begin{proposition}
		\label[proposition]{decomposition}
		Let $C\subseteq S$ be a geometrically integral closed curve on a nice surface. Assume that $C$ is ample on $S$. Let $\pi: \widetilde{C}\rightarrow C$ be the normalisation of $C$. Let $e$ be a positive integer such that the set of rational points on the Brill--Noether locus $W_e \widetilde{C}$ is finite. Let $D\subseteq S$ be an interpolation of degree $e$ points on $C$ (in the sense of \cref{degree-e-interpolation}). Then the set $|C|_e$ of degree $e$ points on $C$ admits a decomposition
		\begin{equation*}
			|C|_e=F\cup \bigcup_{i=1}^m S_i,
		\end{equation*} 
		where $F$ is a finite set, and for each $i$, the set $S_i$ is the set of degree $e$ points obtained by intersecting $C$ with the linear system of curves linearly equivalent to $D$ that pass through a certain effective Weil divisor $B_i\subseteq C$, that is, 
		\begin{equation*}
			S_i = \left\{P\in|C|_e \text{ }\bigg|\text{ } \exists D'\in |D|:P=C\cap D'-B_i\right\}.
		\end{equation*}
		Furthermore, the finite set $F$ is the set of all singular points and regular points of degree $e$ which do not move as Cartier divisors.
	\end{proposition}
	\begin{proof}
		First, as the set $W_e \widetilde{C}(k)$ is finite, the subset
		\begin{equation*}
			\widetilde{F}\colonequals \left\{P\in |\widetilde{C}|_e \text{ }\bigg|\text{ } h^0(\widetilde{C},P)=1\right\}
		\end{equation*}
		of degree $e$ points on $\widetilde{C}$ not moving in a pencil is finite. Let $\Delta \colonequals C\setminus C_{\text{reg}}$ be the finite set of singular points on $C$ and define $F\colonequals |\pi(\widetilde{F})\cup \Delta|_e$. Then the set $F$ is a finite set of degree $e$ points on $C$ containing all the singular degree $e$ points. 
		
		Let $\widetilde{P}_1,\dots ,\widetilde{P}_m\in |\pi^{-1}(C_{\text{reg}})|_e$ be a collection of representatives of every linear equivalence class of degree $e$ points $\widetilde{P}$ such that $\pi(\widetilde{P})\in |C|$ is regular and $\widetilde{P}$ is moving in a pencil. This set is finite by the assumption that $W_e C(k)$ is finite. As the restriction $\pi|_{\pi^{-1}(C_{\text{reg}})}: \pi^{-1}(C_{\text{reg}})\rightarrow C_{\text{reg}}$ is an isomorphism, the points $P_i \colonequals \pi(\widetilde{P}_i)$ are regular degree $e$ points on $C$. Denote
		\begin{equation*}
			S_i=\left\{P\in |C_\text{reg}|_e \text{ }\bigg|\text{ } P\wequiv P_i\right\}
		\end{equation*} 
		the class of degree $e$ points linearly equivalent to $P_i$ as Weil divisors on $C$. We claim that $|C|_e = F\cup \bigcup_{i=1}^m S_i$. Indeed, choose $P\in |C|_e\setminus F$; we have to show that $P\in S_i$ for some $i$. Then $P$ is a regular point on $C$ such that the pullback $\widetilde{P}\colonequals\pi^{*}(P)$ moves. Hence $\widetilde{P}\equiv \widetilde{P}_i$ for some $i$. As rational equivalence is preserved under the proper pushforward $\pi_{*}:\WDiv(\widetilde{C})\rightarrow \WDiv(C)$, we see that $P\wequiv P_i$ and hence $P\in S_i$ as claimed.
		
		The Cartier divisors $P_i$ move in a pencil, since $H^0(\widetilde{C}, \widetilde{P}_i) \cong H^0(C,\pi_{*}\OO_{\widetilde{C}}(\widetilde{P}_i)) = H^0(C,P_i)$. Hence, by \cref{basepoint-free}, the Cartier divisors $P_i$ move in a basepoint-free pencil. Therefore, by the assumption \eqref{lift} in \cref{degree-e-interpolation}, there exists an effective Cartier divisor $B'_i\in \CDiv(C)$ such that 
		\begin{equation*}
			P_i\cequiv D|_C-B'_i.
		\end{equation*}
		Denote by $B_i=\cyc(B'_i)$ the associated effective Weil divisor in $C$.
		
		Let $P\in S_i$ be a point. The Weil divisor $P+B_i$ is effective and linearly equivalent to $\cyc(D|_C)$. Hence there exists an effective Cartier divisor $E\in \CDiv(C)$ linearly equivalent to $D|_C$ such that 
		\begin{equation*}
			P+B_i=\cyc(E).
		\end{equation*}
		But then, by assumptions \eqref{resisinter}, \eqref{surjective} in \cref{degree-e-interpolation} and \cref{varying-intersection}, there exists some curve $D'\subseteq S$ linearly equivalent to $D$ such that $\cyc(E)=C\cap D'$. Hence $P=C\cap D'-B_i$. Thus 
		\begin{equation*}
			S_i = \left\{P\in|C|_e \text{ }\bigg|\text{ } \exists D'\in |D|:P=C\cap D'-B_i\right\}.
		\end{equation*} 
		as claimed.
	\end{proof}
	
	\section{Toric surfaces and an inequality of Martens}
	
	Before we prove \cref{maintoric}, we review some properties of toric surfaces. All the tools needed can be found in \cite{cox2011toric}. Let $S$ be a nice toric surface. Denote by $D_1,\dots, D_n\subseteq S$ the prime toric divisors. By \cite[Theorem 4.1.3]{cox2011toric} the group $\Pic(S)=\text{NS}(S)$ is generated by $D_1,\dots, D_n$, but these generators might have relations. Denote by $K=K_S$ the canonical divisor of $S$. By \cite[Theorem 8.2.3]{cox2011toric}, we have
	\begin{equation*}
		K\equiv -\sum_{i=1}^{n} D_i.
	\end{equation*} 
	Recall that a toric $\Q$-divisor $D\in \Q \otimes \Div(S)$ induces a polytope $P_D$ in $\R^2$ which is full-dimensional when $D$ is ample, for example by \cite[Theorem 7.2.10]{cox2011toric}. For a general $\Q$-divisor $D\in \Q \otimes \Div(S)$ (not necessarily toric), we use the notation $\lfloor D\rfloor$ to denote the toric divisor on $S$ which is constructed as follows. Choose a toric representation $D\equiv \sum_{i=1}^n a_iD_i$ and define 
	\begin{equation*}
		\lfloor D\rfloor = \sum_{i=1}^n \lfloor a_i\rfloor D_i
	\end{equation*}
	and similarly, $\lceil D\rceil = \sum_{i=1}^n \lceil a_i\rceil D_i$. When $D\subseteq S$ is a closed curve, by \cite[Corollary 6.3.21]{cox2011toric} there exists such toric representation $D\equiv \sum_{i=1}^n a_iD_i$ where the coefficients $a_i\geq0$ are non-negative. Thus, as a convention, when $D\in \Q \otimes \Div(S)$ is a rescaling of a closed curve in $S$, we always choose such non-negative coefficients for our toric representation of $D$. The definition of $\lfloor D\rfloor$ and $\lceil D\rceil$ depends on the choice of toric representation, which is fine for our needs. We will use the following vanishing theorem:
	
	\begin{theorem}[\cite{cox2011toric}, Theorem 9.3.5]
		\label{vanishing}
		Let $S$ be a nice toric surface. Let $D\in \Q \otimes \Div(S)$ be a nef $\Q$-divisor on $S$. Then
		\begin{enumerate}
			\item For all $i>0$ the cohomology $H^i(S,\lfloor D\rfloor)=0$ vanishes.
			\item For all $i\neq \dim P_D$ the cohomology $H^i(S,-\lceil D\rceil)=0$ vanishes. In particular when $D$ is ample, for $i=0,1$ the cohomologies $H^i(S,-\lceil D\rceil)=0$ vanishes.
		\end{enumerate}
	\end{theorem}
	
	With the above notation, we define the constant $\lambda(S)$ from \cref{maintoric} in the following way.
	\begin{definition}
		Let $S$ be a nice toric surface. We define $\lambda(S)$ to be
		\begin{equation*}
			\lambda(S)\colonequals 2+\frac{1}{4} \min_{R\subseteq [n]}(\sum_{i\in R}D_i)\cdot (2K+\sum_{i\in R}D_i)
		\end{equation*}
	\end{definition}
	
	This invariant is generally easy to compute, for example:
	\begin{example}
		\label[example]{lambdap2}
		Let $S=\PP_k^2$ be the projective plane. Then there are $3$ prime toric divisors. All of them are lines, in particular linearly equivalent to some fixed line $H\subseteq S$. Recall that $K=-3H$. Hence, 
		\begin{equation*}
			\min_{R\subseteq [n]}(\sum_{i\in R}D_i)\cdot (2K+\sum_{i\in R}D_i)=\min_{R\subseteq [3]}(|R|H\cdot (|R|-6)H)=\min_{0\leq r\leq 3}r(r-6)=-9
		\end{equation*}
		and thus $\lambda(S)=-\frac{1}{4}$.
	\end{example}
	
	\begin{example}
		Let $S=F_m$ be the $m$-th Hirzebruch surface, see \cite[Example 3.1.16]{cox2011toric} for a definition. Let $C_0,F\subseteq S$ be a section and a fibre, respectively. Then there are $4$ prime toric divisors $D_i$ on $S$, which are linearly equivalent to
		\begin{equation*}
			F,\text{   }C_0,\text{   }F,\text{   }C_0+mF.
		\end{equation*}
		One can check the $16$ possibilities of $(\sum_{i\in R}D_i)\cdot (2K+\sum_{i\in R}D_i)$, and see that $\lambda(S) = -\frac{m}{4}$.
	\end{example}
	
	We will also use the following lemma due to Martens \cite[Observation 2.4]{Martens1968}; we include the proof for completeness.
	\begin{lemma}
		\label[lemma]{bpfp-lemma}
		Let $C$ be a geometrically integral proper curve. Let $Z,E\in \CDiv(C)$ be two Cartier divisors on $C$. Assume that $E$ moves in a basepoint-free pencil. Then 
		\begin{equation*}
			h^0(Z+E)+h^0(Z-E)\geq 2h^0(Z)
		\end{equation*}
	\end{lemma}
	\begin{proof}
		By passing to linearly equivalent Cartier divisors, we may assume that $Z,E$ are supported on $C_\text{reg}$. Let $\pi:\widetilde{C}\rightarrow C$ be the normalization of $C$, then for all Cartier divisors $D\in \CDiv(C)$ which are supported on $C_\text{reg}$, we have an isomorphism
		\begin{equation*}
			H^0(\widetilde{C},\pi^*D)\cong H^0(C,\pi_{*}\OO_{\widetilde{C}}(\pi^*D))=H^0(C,D).
		\end{equation*}
		As the pullback $\pi^*E\in \Div(\widetilde{C})$ moves in a basepoint-free pencil, by passing to $\widetilde{C}$ we may assume that $C$ is a nice curve. 
		
		We view the space of global section $H^0(D)$ associated to a divisor as a subspace of $k(C)$, the field of rational functions on $C$. Let $E',E''\in \Div(C)$ be two effective divisors linearly equivalent to $E$ which share no points in their supports. Then 
		\begin{equation*}
			H^0(Z)\cap H^0(Z+E'-E'') = H^0(Z-E'')
		\end{equation*}
		Furthermore, $H^0(Z)\subseteq H^0(Z+E')$ and $H^0(Z+E'-E'')\subseteq H^0(Z+E')$. Using the dimension formula for $H^0(Z) + H^0(Z+E'-E'')$, we see that 
		\begin{align*}
			h^0(Z+E') &\geq \dim(H^0(Z) + H^0(Z+E'-E'')) \\ 	   	
					  &= h^0(Z)+h^0(Z+E'-E'')-h^0(Z-E'')
		\end{align*}
		Taking into account that $E''\equiv E'\equiv E$, we conclude the inequality $h^0(Z+E)+h^0(Z-E)\geq 2h^0(Z)$.
	\end{proof}
	\section{Proof of the main theorem on toric surfaces}
	
	We begin with a cohomological calculation.
	
	\begin{lemma}
		\label[lemma]{ec}
		Let $S$ be a nice toric surface and let $C\subseteq S$ be a geometrically integral ample curve on it. Let $D=\lfloor \frac{C}{2}\rfloor \in \Div(S)$. Then the Euler characteristic of $D$ is $\chi(S,D)=h^0(C,D|_C)$.
	\end{lemma}
	\begin{proof}
		By \cref{vanishing}, we see that $h^i(S,D)=0$ for $i=1,2$, hence the Euler characteristic of $D$ is 
		\begin{equation}
			\label{euler-char-d}
			\chi(S,D)=h^0(S,D).
		\end{equation}
		Consider the $D$-twisted ideal sheaf short exact sequence of $C$ in $S$,
		\begin{equation*}
			0\rightarrow \OO_S(D-C)\rightarrow \OO_S(D)\rightarrow \OO_C(D|_C)\rightarrow 0.
		\end{equation*}
		Taking cohomology, we obtain the following exact sequence:
		\begin{equation*}
			H^0(S,D-C)\rightarrow H^0(S,D)\rightarrow H^0(C,D|_C)\rightarrow H^1(S,D-C).
		\end{equation*}
		But $D-C=-\lceil\frac{C}{2}\rceil$ and by \cref{vanishing}, we see that $H^0(S,D-C)=H^1(S,D-C)=0$, hence $H^0(S,D)\rightarrow H^0(C,D|_C)$ is an isomorphism. Combining with \eqref{euler-char-d}, we see that the Euler characteristic of $D$ is the dimension of the global sections of $D|_C$,
		\begin{equation*}
			\chi(S,D)=h^0(C,D|_C). \qedhere
		\end{equation*}
	\end{proof}
	
	We now introduce the main geometric ingredient needed for the proof of \cref{maintoric}. The proof below is inspired by the work of Coppens and Kato \cite{Coppens1991TheGO} on projective plane curves. 
	
	\begin{proposition}
		\label[proposition]{mainprop}
		Let $S$ be a nice toric surface and let $C\subseteq S$ be a geometrically integral ample curve on it. Let $D=\lfloor \frac{C}{2}\rfloor \in \Div(S)$. Then for all Cartier divisors $E\in \CDiv(C)$ moving in a basepoint-free pencil with degree satisfying 
		\begin{equation*}
			\deg E < \frac{C^2}{4} + \lambda (S)
		\end{equation*}
		the cohomology $H^0(C,D|_C-E)\neq 0$ does not vanish.
	\end{proposition}
	\begin{proof}
		Let $E\in \CDiv(C)$ be a Cartier divisor moving in a basepoint-free pencil, of degree $\deg E=e$ satisfying $e < \frac{C^2}{4} + \lambda (S)$. Assume by way of contradiction that $H^0(C,D|_C-E)=0$. 
		
		By \cite[Chapter I Propositions 2.3, 2.4]{altman1970introduction} the dualising sheaf of $C$ is a line bundle which corresponds to the Cartier divisor class of $K_C=(K+C)|_C$. By the Riemann-Roch theorem on $C$, 
		\begin{equation}
			\label{h0}
			h^0\colonequals h^0((K+C-D)|_C-E) = h^0(K_C-(D|_C + E)) = p_a-1 - D\cdot C - e +h^0(D|_C+E)
		\end{equation}
		where $p_a$ is the arithmetic genus of $C$. By \cref{bpfp-lemma} applied to $D|_C$ and using the vanishing of $h^0(D|_C-E)$, we see that $h^0(D|_C+E)\geq 2h^0(D|_C)$.
		By \cref{ec}, we conclude that
		\begin{equation}
			\label{ineq}
			h^0(D|_C+E)\geq 2\chi(S,D)
		\end{equation}
		Using the Hirzebruch-Riemann-Roch theorem, the Euler characteristic $\chi(S,D)$ is 
		\begin{equation}
			\label{hirzebruch}
			\chi(S,D) = \chi(S,0)+\frac{D^2-K\cdot D}{2} = 1 + \frac{D^2-K\cdot D}{2};
		\end{equation}
		here we used the toricness of $S$ to obtain $\chi(S,0)=1$. Applying \eqref{ineq} and \eqref{hirzebruch} to \eqref{h0}, we obtain the inequality 
		\begin{equation*}
			h^0 \geq p_a-1 - D\cdot C - e + D^2-K\cdot D + 2
		\end{equation*}
		Note that $p_a-1=\frac{\deg(K_C)}{2}=\frac{(K+C)\cdot C}{2}$ and that
		\begin{equation*}
			D^2=\frac{1}{4}(2D-C)\cdot(2D+C)+\frac{C^2}{4}.
		\end{equation*}
		Hence 
		\begin{align*}
			h^0 &\geq \frac{1}{2}(K+C)\cdot C - D\cdot (K+C) +\frac{1}{4}(2D-C)\cdot(2D+C)+2+\frac{C^2}{4}-e \\
			&= \frac{1}{2}(K+C)\cdot (C-2D) + \frac{1}{2}(C-2D)\cdot(-\frac{2D+C}{2})+2+\frac{C^2}{4}-e \\
			&= \frac{1}{2}(C-2D)\cdot (K+C-\frac{2D+C}{2}) + 2+\frac{C^2}{4}-e\\
			&= \frac{1}{4}(C-2D)(2K+C-2D)+ 2+\frac{C^2}{4}-e.
		\end{align*}
		Now, we note that the divisor $C-2D$ is of the form $\sum_{i\in R}D_i$ for some subset $R\subseteq [n]$. Hence, by definition, 
		\begin{equation*}
			\frac{1}{4}(C-2D)(2K+C-2D)+ 2\geq \lambda (S).
		\end{equation*}
		Thus, $h^0\geq \frac{C^2}{4}+\lambda(S)-e$, which by assumption is greater than $0$. Finally, note that $K+C-D\leq D$, so
		\begin{equation*}
			h^0(D|_C-E)\geq h^0((K+C-D)|_C-E)=h^0>0
		\end{equation*}
		which is equivalent to the non-vanishing of $H^0(C,D|_C-E)$.
	\end{proof}		
	
	We are now ready to prove \cref{maintoric}.
	
	\begin{theorem}
		\label{maintoric-proof}
		Let $S$ be a nice toric surface. Let $C\subseteq S$ be a geometrically integral closed curve with simple singularities of multiplicities $\delta_1,\dots, \delta_n$ such that $C+K>0$. Assume furthermore that the normalisation of $C$ is ample on the blowup of $S$ at the singular points of $C$. Then there exists a closed curve $D\subseteq S$ with $C\cdot D\leq \frac{C^2}{2}$, and such that for all positive integers
		\begin{equation*}
			e<\min \left(\frac{C^2-\sum_{i=1}^{n}\delta_i^2}{9}, \frac{C^2}{4}+\lambda(S)\right),
		\end{equation*}
		the set $|C|_e$ of degree $e$ points on $C$ admits the following decomposition
		\begin{equation*}
			|C|_e=F\cup \bigcup_{i=1}^m S_i,
		\end{equation*} 
		where $F$ is a finite set, and for each $i$, the set $S_i$ is the set of degree $e$ points obtained by intersecting $C$ with the linear system of curves linearly equivalent to $D$ that pass through a certain effective Weil divisor $B_i\subseteq C$, that is, 
		\begin{equation*}
			S_i = \left\{P\in|C|_e \text{ }\bigg|\text{ } \exists D'\in |D|:P=C\cap D'-B_i\right\}.
		\end{equation*}
		Furthermore, the finite set $F$ is the set of all singular points and regular points of degree $e$ which do not move as Cartier divisors.
	\end{theorem}
	
	\begin{proof}
		Let $\widetilde{C}\subseteq \Bl(S)$ be the normalisation of $C$ embedded on the blowup of $S$ along the singularities of $C$. Note that by assumption $\widetilde{C}$ is ample on $\Bl(S)$ and in particular $C$ is ample on $S$. 
		
		By the assumption $C+K>0$ and \cite[Corollary 6.3.21]{cox2011toric} we may choose a toric representation 
		\begin{equation*}
			C\equiv \sum_{i=1}^n  a_i D_i
		\end{equation*}
		such that $a_i\geq 1$ for all $i\in [n]$ and $a_j\geq 2$ for at least one $j\in [n]$. Take $D = \lfloor \frac{C}{2}\rfloor \in \Div(S)$ with the above toric representation of $C$. Then $D$ is a nonzero effective divisor, hence a closed curve $D\subseteq S$. By definition, as divisors $2D\leq C$. Taking the intersection product with $C$ and taking into account the ampleness of $C$, we conclude that $C\cdot D\leq \frac{C^2}{2}$.
		
		Let $e < \min (\frac{C^2-\sum_{i=1}^{n}\delta_i^2}{9}, \frac{C^2}{4}+\lambda(S))$. As $S$ is toric, it has irregularity $h^1(S,\OO_S)=0$. Hence, by \cref{singsmithvogt}, the Brill--Noether locus $W_e\widetilde{C}_{\overline{k}}$ contains no positive-dimensional abelian varieties. Applying the Mordell-Lang conjecture (proven by Faltings \cite{FALTINGS1994175}), we see that the set of rational points $W_e\widetilde{C}(k)$ is finite. 
		
		We now claim that $D$ is an interpolation of degree $e$ points on $C$; we have to verify the three conditions of \cref{degree-e-interpolation}. By the ampleness of $C$, we see that $C\cdot D \leq \frac{C^2}{2} < C^2$, which is condition \eqref{resisinter}. By \cref{vanishing} and the observation that $D-C=-\lceil\frac{C}{2}\rceil$ we see that $H^1(S,D-C)=0$ and in particular the map
		\begin{equation*}
			H^0(S,D)\rightarrow H^0(C,D|_C)
		\end{equation*}
		is surjective, this is condition \eqref{surjective}. By \cref{mainprop}, for all degree $e$ Cartier divisors $E\in \CDiv(C)$, the divisor $D|_C-E$ has global sections, which is condition \eqref{lift}. Thus, the conditions of \cref{decomposition} are satisfied; applying it gives the desired decomposition.
	\end{proof}
	
	As promised in \cref{blample}, we give an example of an explicit sufficient condition for $\widetilde{C}$ to be ample on $\Bl(S)$.
	 
	\begin{example}
		\label[example]{ample}
		Let $S$ be a nice toric surface and let $C\subseteq S$ be an integral ample curve. Denote by $P_1,\dots, P_n$ the singular points of $C$, denote by $\delta_i$ the multiplicity $P_i$ on $C$. Denote $r=\min(C\cdot D_i)$, and assume that there exist positive numbers $\alpha_i$ such that 
		\begin{equation*}
			\frac{\delta_i}{\alpha_i}<r,\text{   }\sum_{i=1}^{n}\alpha_i = 1.
		\end{equation*}
		We claim that $\widetilde{C}$ is ample on the blowup $\pi: \Bl(S)\rightarrow S$ of $S$ along the singular points of $C$. By the main result of \cite{Rocco1997GenerationOK}, for all $x\in S$ the Seshadri constant of $\OO_S(C)$ at $x$ satisfies
		\begin{equation*}
			\epsilon(S,\OO_S(C);x)\geq r.
		\end{equation*}
		Hence, for all $i\leq n$, the $\R$-divisor $\alpha_i \pi_i^*C-\delta_i E_i$ is ample on the blowup $\pi_i :\Bl_{P_i}(S)\rightarrow S$ of $S$ along $P_i$. Pulling back to $\Bl(S)$, we see that the divisors $\alpha_i \pi^*C-\delta_i E_i$ are ample on $\Bl(S)$. As a sum of ample divisors is ample, we conclude that 
		\begin{equation*}
			\widetilde{C} \equiv \pi^*C-\sum_{i=1}^n \delta_i E_i
		\end{equation*}
		is ample on $\Bl(S)$.
		
		In particular, for $S=\PP_k^2$ and $C$ of degree $d$ with $\delta$  ordinary nodes and cusps such that $2\delta < d$, the normalisation $\widetilde{C}$ is ample on the blowup $\Bl(S)$. Indeed, in this case $r=d$, the multiplicities are $\delta_i=2$ and we may take $\alpha_i = \frac{1}{\delta}$. We will use this case in the next section.
	\end{example}	
	
	\section{Plane curves}
	As discussed in \cref{basebig}, we would like a version of \cref{maintoric} for which the degrees $\deg(B_i) = C\cdot D - e$ are minimised. Along with the benefit of $B_i$ being simpler to understand, for some degrees $e'$, this also can help prove that there are only finitely many linear equivalence classes of degree $e'$ points, as done in \cite{10.1093/imrn/rnaa137} and in \cref{no-avs-plane}. We first justify studying curves on $\PP_k^2$ rather than on general nice toric surfaces. The following example shows that the natural analogue of Debarre--Klassen fails for curves on Hirzebruch surfaces. 
	
	\begin{example}
		\label[example]{mainexample}
		Let $S=F_1$ be the first Hirzebruch surface. Let $C_0,F\subseteq S$ be a section and a fibre, respectively. By \cite[Chapter V Proposition 2.3]{hartshorne1977algebraic}, the N\'eron-Severi group of $S$ is $\text{NS}(S)\cong \Z\oplus\Z$ and generated by $C_0,F$ which satisfy 
		\begin{equation*}
			C_0^2=-1,\text{   }C_0\cdot F=1,\text{   }F^2=0.
		\end{equation*}
		By \cite[Chapter V Corollary 2.18]{hartshorne1977algebraic}, for all $n\in \Z_{>0}$ there exists a nice very ample curve $C_n \subseteq S$ of linear equivalence class $nC_0+(n+1)F$. By the same argument, let $D\subseteq S$ be a very ample curve of linear equivalence class $C_0+3F$. Since $D$ is very ample on $S$, the divisor $D|_{C_n}$ moves in a basepoint-free pencil. Hence, by Hilbert's irreducibility theorem, there exists a closed point $P\in |C_n|$ linearly equivalent to $D|_{C_n}$. Note that
		\begin{equation*}
			C^2_n = n^2+2n,\text{   } \deg P = D\cdot C_n = 3n+1
		\end{equation*}
		so for big enough $n$, the inequality $\deg P<\frac{C_n^2}{9}$ holds and the point $P$ should be considered as a low degree point. Fix some big $n$ and denote $C=C_n$. A cohomological calculation shows that the cohomology $H^1(S,D)=0$ vanishes, and $H^1(S,D-C)\neq 0$ does not vanish. Consider the $D$-twisted ideal sheaf short exact sequence of $C$ in $S$,
		\begin{equation*}
			0\rightarrow \OO_S(D-C)\rightarrow \OO_S(D)\rightarrow \OO_C(P)\rightarrow 0.
		\end{equation*}
		Taking the long exact sequence of cohomology gives the following exact sequence
		\begin{equation*}
			H^0(S,D-C)\rightarrow H^0(S,D)\rightarrow H^0(C,P)\rightarrow H^1(S,D-C)\rightarrow H^1(S,D).
		\end{equation*}
		But $H^0(S,D-C)=H^1(S,D)=0$ and hence by exactness 
		\begin{equation*}
			h^0(C,P)=h^0(S,D)+h^1(S,D-C)
		\end{equation*}
		and as $H^1(S,D-C)\neq 0$ does not vanish, we see that $h^0(C,P)>h^0(S,D)$. Hence, the map $H^0(S,D)\rightarrow H^0(C,P)$ cannot be surjective. Thus, by Hilbert's irreducibility theorem, there exists a point $P'\in |C|$ which is linearly equivalent to $D|_C$ but not of the form $C\cap D'$ for all curves $D'\subseteq S$ which are linearly equivalent to $D$.  
	\end{example}
	
	In the projective plane $\PP_k^2$, due to $H^1(\PP_k^2,-)$ being the $0$ functor on line bundles, the issues imposed by the non-vanishing of $H^1(S,D-C)$ disappear. In this section, we prove \cref{plane}. The main tools needed are explicit descriptions of some pencils on plane curves due to Coppens and Kato. 
	
	\begin{lemma}[\cite{Coppens1991TheGO}, Main Lemma]
		\label[lemma]{mainlem}
		Let $\widetilde{C}$ be the normalisation of an integral plane curve $C$ of degree $d$ with $\delta$ ordinary nodes and cusps as its singularities. Let $E\in \Div \widetilde{C}$ be an effective divisor of degree $e$ moving in a basepoint free pencil, such that $e+\delta < (m+1)(d-(m+1))$ for some integer $m$. Then the cohomology $H^0(\widetilde{C}, m\widetilde{H}-E)\neq 0$ does not vanish, where $\widetilde{H}$ is a pullback of a line $H\subseteq \PP_{k}^{2}$ under the map $\widetilde{C}\rightarrow C\hookrightarrow \mathbb{P}^2_k$. 
	\end{lemma}
	
	\begin{theorem}[\cite{Coppens1991TheGO}, Theorem 2.1]
		\label{gonality}
		Let $\widetilde{C}$ be the normalisation of an integral plane curve $C$ of degree $d\geq 4$ with $\delta$ ordinary nodes and cusps as its singularities. Suppose that there exists an integer $m$ such that 
		\begin{equation*}
			d\geq 2m\hspace{.5cm}\text{and}\hspace{.5cm} \delta < (m-1)d-m^2+3
		\end{equation*}
		then $\widetilde{C}$ has no linear systems $g^1_{d-3}$.
	\end{theorem}
	
	As a corollary, we deduce the following inequality:
	\begin{corollary}
		\label[corollary]{ineq-plane}
		Let $\widetilde{C}$ be the normalisation of an integral plane curve $C$ of degree $d$ with $\delta \leq \frac{d-3}{3}$ ordinary nodes and cusps as its singularities. Let $E\in \Div(\widetilde{C})$ be a divisor moving in a pencil. Then
		\begin{equation*}
			d-1 \leq \deg(E)+\delta.
		\end{equation*}
	\end{corollary}
	\begin{proof}
		When $C$ is smooth, that is $\delta=0$, the curve $\widetilde{C}$ is equal to $C$, which is a smooth plane curve of degree $d$. By Noether's gonality theorem $d-1\leq \deg(E)=\deg(E)+\delta$. Otherwise, the number of singularities $\delta$ is at least  $1$. Applying \autoref{gonality} to $m=2$, we see that $d-2\leq \deg(E)$. Thus $d-1\leq \deg(E)+\delta$.
	\end{proof}
	
	We will also need the following lemma, first observed by Abramovich and Harris \cite{Abramovich1991AbelianVA}. 
	\begin{lemma}[\cite{10.1093/imrn/rnaa137}, Lemma 2.1]
		\label[lemma]{A2}
		Let $C$ be a nice curve and let $e$ be an integer such that $W_{e-1} C_{\overline{k}}$ does not contain positive-dimensional abelian translates. Let $A\subseteq W_e C_{\overline{k}}$ be a positive dimensional abelian translate and let $A_2$ be the image of $A\times A$ under the addition map $W_e C_{\overline{k}}\times W_e C_{\overline{k}}\rightarrow W_{2e} C_{\overline{k}}$. Then all the line bundles associated to points of $A_2$ move in a basepoint-free pencil. 
	\end{lemma}
	
	For the rest of this section, let $C\subseteq\mathbb{P}_{k}^{2}$ be a geometrically integral plane curve of degree $d$ with $\delta$ ordinary nodes and cusps as its singularities. Let $\widetilde{C}\rightarrow C$ be the normalisation of $C$. Then $\widetilde{C}$ is a nice curve over $k$. We begin with the following claim:
	\begin{lemma}
		\label[lemma]{no-avs-plane}
		Assume that $\delta \leq \frac{d-3}{3}$. Then for all positive integers 
		\begin{equation*}
			e < \max\left(\frac{d^2-4\delta}{9}, \text{   } \frac{1}{2}\left(\left\lceil \frac{d+\sqrt{d^2-36\delta}}{6}\right\rceil\left(d-\left\lceil \frac{d+\sqrt{d^2-36\delta}}{6}\right\rceil\right)-\delta\right)\right)
		\end{equation*} 
		the variety $W_e \widetilde{C}_{\overline{k}}$ contains no positive dimensional abelian translates. In particular, the set of rational points $W_e \widetilde{C}(k)$ is finite.
	\end{lemma}
	\begin{proof}
		We first prove the claim for $e < \frac{d^2-4\delta}{9}$. Note that $2\delta < d$ and hence, by \cref{ample}, the curve $\widetilde{C}$ is ample on the blowup $\Bl(S)$ of $S$ along the singular points of $C$. Thus, by \cref{singsmithvogt}, we conclude that there are no positive dimensional abelian translates contained in $W_e \widetilde{C}_{\overline{k}}$ for all positive integers $e < \frac{d^2-4\delta}{9}$.
				
		Now, we prove the claim for $e < \frac{1}{2}\left(\left\lceil \frac{d+\sqrt{d^2-36\delta}}{6}\right\rceil\left(d-\left\lceil \frac{d+\sqrt{d^2-36\delta}}{6}\right\rceil\right)-\delta\right)$. Suppose by way of contradiction that there exists such $e$ for which $W_e \widetilde{C}_{\overline{k}}$ contains a positive dimensional abelian translate; choose $e$ to be minimal. Let $A\subseteq W_e \widetilde{C}_{\overline{k}}$ be a positive dimensional abelian translate and let $A_2\subseteq W_{2e} \widetilde{C}_{\overline{k}}$ be the image of $A\times A$ under the addition map $W_e \widetilde{C}_{\overline{k}}\times W_e \widetilde{C}_{\overline{k}}\rightarrow W_{2e} \widetilde{C}_{\overline{k}}$. Let $p\in A_2$ be a point, by \cref{A2} as a line bundle it moves in a pencil and thus by \cref{ineq-plane}
		\begin{equation*}
			d-1\leq 2e+\delta < \lceil \frac{d+\sqrt{d^2-36\delta}}{6}\rceil(d-\lceil \frac{d+\sqrt{d^2-36\delta}}{6}\rceil).
		\end{equation*}
		Since the function $t(d-t)$ is increasing when $t\in [0,\frac{d}{2}]$, there exists a positive integer $m<\frac{d+\sqrt{d^2-36\delta}}{6}$ such that 
		\begin{equation}
			\label{2e}
			m(d-m)\leq 2e+\delta < (m+1)(d-(m+1)).
		\end{equation}
		Hence, by \cref{mainlem} we obtain a morphism 
		\begin{equation*}
			A_2\rightarrow W_{md-2e}\widetilde{C}_{\overline{k}}
		\end{equation*} 
		given by $p\mapsto m\widetilde{H}-p$, where $\widetilde{H}$ is a pullback of a line $H\subseteq \PP_{k}^{2}$. This morphism is injective, hence $W_{md-e}\widetilde{C}$ contains an abelian translate. By the minimality of $e$, we conclude that $md-2e\geq e$. But by \eqref{2e},
		\begin{equation*}
			md-2e\leq md-m(d-m)+\delta = m^2+\delta.
		\end{equation*}
		We explicitly verify that the quadratic inequality $m^2+\delta<\frac{1}{2}(m(d-m)-\delta)$ on $m$ holds when $1\leq m  < \frac{d+\sqrt{d^2-36\delta}}{6}$, here we use the assumption $\delta \leq \frac{d-3}{3}$ (whence square roots in the bound on $e$). Overall using \eqref{2e} again we conclude that
		\begin{equation*}
			md-2e<e.
		\end{equation*}
		This is a contradiction, so there are indeed no positive dimensional abelian translates in $W_e \widetilde{C}_{\overline{k}}$.
		
		By Lang's conjecture (\cref{lang}), we conclude that the set of rational points $W_e \widetilde{C}(k)$ is finite.
	\end{proof}
	
	We are now ready to prove \cref{plane}.
	
	\begin{theorem}
		\label{plane-proof}
		Let $C\subseteq\mathbb{P}_{k}^{2}$ be a geometrically integral plane curve of degree $d\geq 4$ with $\delta \leq \frac{d-3}{3}$ ordinary nodes and cusps as its singularities. Let
		\begin{equation*}
			e< \max\left(\frac{d^2-4\delta}{9}, \text{   } \frac{1}{2}\left(\left\lceil \frac{d+\sqrt{d^2-36\delta}}{6}\right\rceil\left(d-\left\lceil \frac{d+\sqrt{d^2-36\delta}}{6}\right\rceil\right)-\delta\right)\right)
		\end{equation*}
		be a positive integer. Then the set of degree $e$ points on $C$ admits a decomposition
		\begin{equation*}
			|C|_e=F\cup\bigcup_{i=1}^{n}S_i,
		\end{equation*}
		where $F$ is a finite set, and for each $i$, the set $S_i$ is the set of degree $e$ points obtained by intersecting $C$ with the linear system of degree $m$ curves passing through a certain effective Weil divisor $B_i\subseteq C$. The finite set $F$ is the set of all singular points and regular points which do not move as Cartier divisors. In case that $S_i$ are nonempty, the degree $m$ is a positive integer such that
		\begin{equation*}
			m(d-m)\leq e+\delta < (m+1)(d-(m+1)),
		\end{equation*}
		and in particular the divisors $B_i$ are of degree $\deg B_i = md-e < \frac{e}{2}$.
	\end{theorem}
	
	\begin{proof}
		As the degree of $C$ is positive, we see that $C$ is ample on $\PP^2_k$. Let $\pi:\widetilde{C}\rightarrow C$ be the normalisation of $C$. By \cref{no-avs-plane}, the set of rational points on the Brill--Noether locus $W_e\widetilde{C}$ is finite. If all regular degree $e$ points $P\in |C_\text{reg}|_e$ do not move as Cartier divisors, we conclude that the set of degree $e$ points $|C|_e$ is finite. Otherwise, there exists a degree $e$ regular point $P\in |C_\text{reg}|_e$ which moves. By \cref{basepoint-free}, it moves in a basepoint-free pencil. Let $\widetilde{P}\colonequals \pi^{*}(P)$ be the unique closed point in $\widetilde{C}$ above $P$. Then $\widetilde{P}$ is a degree $e$ divisor on $\widetilde{C}$ moving in a basepoint-free pencil. Thus by \cref{ineq-plane},
		\begin{equation*}
			d-1\leq e+\delta < \lceil \frac{d+\sqrt{d^2-36\delta}}{6}\rceil(d-\lceil \frac{d+\sqrt{d^2-36\delta}}{6}\rceil);
		\end{equation*}
		note that the second inequality is clear when $e<\frac{1}{2}\left(\left\lceil \frac{d+\sqrt{d^2-36\delta}}{6}\right\rceil\left(d-\left\lceil \frac{d+\sqrt{d^2-36\delta}}{6}\right\rceil\right)-\delta\right)$; when $e<\frac{d^2-4\delta}{9}$ it follows from our assumption that $\delta \leq \frac{d-3}{3}$. Since the function $t(d-t)$ is increasing when $t\in [0,\frac{d}{2}]$, there exists a positive integer $1\leq m<\lceil\frac{d+\sqrt{d^2-36\delta}}{6}\rceil$ such that 
		\begin{equation}
			\label{e}
			m(d-m)\leq e+\delta < (m+1)(d-(m+1));
		\end{equation}
		note that as $m$ is an integer we must have $m<\frac{d+\sqrt{d^2-36\delta}}{6}$. Let $H\subseteq \PP_{k}^2$ be a line not passing through any singular point of $C$, let $D=mH$. We now claim that $D$ is an interpolation of degree $e$ points on $C$; we have to verify the three conditions of \cref{degree-e-interpolation}. Note that $m<\frac{d}{2}$, so $C\cdot D = md<d^2=C^2$, this is condition \eqref{resisinter}. Due to the vanishing of $H^1(\PP_k^2,-)$, the map
		\begin{equation*}
			H^0(\PP^2_k,D)\rightarrow H^0(C,D|_C)
		\end{equation*}
		is surjective, this is condition \eqref{surjective}. Let $E\in \CDiv(C)$ be a degree $e$ Cartier divisor on $C$ which moves in a basepoint-free pencil. Let $E'\in \CDiv(C)$ be a divisor linearly equivalent to $E$ which is supported on $C_{\text{reg}}$. Then 
		\begin{equation*}
			H^0(C, D|_C-E)\cong H^0(C, D|_C-E') = H^0(C,\pi_{*}\OO_{\widetilde{C}}(m\widetilde{H}-\pi^{*}E'))\cong H^0(\widetilde{C}, m\widetilde{H}-\pi^{*}E'),
		\end{equation*}
		where $\widetilde{H}$ is the pullback of $H$ under $\widetilde{C}\rightarrow C\rightarrow \PP^2_k$. Note that $\pi^{*}E'\in \Div(\widetilde{C})$ is a degree $e$ divisor moving in a basepoint-free pencil. Hence, by \cref{mainlem}, we conclude that the divisor $D|_C-E$ has global sections, which is condition \eqref{lift}. Thus, the conditions of \cref{decomposition} are satisfied; applying it gives the desired decomposition.
		
		Finally, by \eqref{e} the degree of the divisors $B_i$ satisfies 
		\begin{equation*}
			\deg B_i=md-e<m^2+\delta.
		\end{equation*}
		Since $0 < m < \frac{d+\sqrt{d^2-36\delta}}{6}$ and by the assumption $\delta \leq \frac{d-3}{3}$, we have $m^2+\delta<\frac{1}{2}(m(d-m)-\delta)$ and overall using \eqref{e} again we conclude that $\deg B_i < \frac{e}{2}$.
	\end{proof}
	
	\bibliographystyle{alpha}
	\bibliography{bibliography}

\begin{thebibliography}{CLS11}

\bibitem[AH91]{Abramovich1991AbelianVA}
Dan Abramovich and Joe~W. Harris.
\newblock Abelian varieties and curves in {$W_d (C)$}.
\newblock {\em Compositio Mathematica}, 78:227--238, 1991.

\bibitem[AK70]{altman1970introduction}
A.~Altman and S.~Kleiman.
\newblock {\em Introduction to Grothendieck duality theory}.
\newblock Lecture notes in mathematics. 1970.

\bibitem[CK91]{Coppens1991TheGO}
Marc R.~M. Coppens and Takaomi Kato.
\newblock The gonality of smooth curves with plane models.
\newblock {\em manuscripta mathematica}, 70:5--25, 1991.

\bibitem[CLS11]{cox2011toric}
D.A. Cox, J.B. Little, and H.K. Schenck.
\newblock {\em Toric Varieties}.
\newblock Graduate studies in mathematics. American Mathematical Soc., 2011.

\bibitem[DK92]{Klassen1992PointsOL}
Olivier Debarre and Matthew~J. Klassen.
\newblock Points of low degree on smooth plane curves.
\newblock {\em Journal f{\"u}r die reine und angewandte Mathematik (Crelles
  Journal)}, 1994:81 -- 88, 1992.

\bibitem[Fal91]{a07a8b9c-f9ee-30b2-8d26-66e0414778da}
Gerd Faltings.
\newblock Diophantine approximation on abelian varieties.
\newblock {\em Annals of Mathematics}, 133(3):549--576, 1991.

\bibitem[Fal94]{FALTINGS1994175}
Gerd Faltings.
\newblock The general case of s. lang's conjecture.
\newblock In Valentino Cristante and William Messing, editors, {\em Barsotti
  Symposium in Algebraic Geometry}, volume~15 of {\em Perspectives in
  Mathematics}, pages 175--182. Academic Press, 1994.

\bibitem[Ful84]{fulton1984intersection}
W.~Fulton.
\newblock {\em Intersection Theory}.
\newblock Ergebnisse der Mathematik und ihrer Grenzgebiete : a series of modern
  surveys in mathematics. Folge 3. Springer-Verlag, 1984.

\bibitem[Har77]{hartshorne1977algebraic}
R.~Hartshorne.
\newblock {\em Algebraic Geometry}.
\newblock Graduate Texts in Mathematics. Springer, 1977.

\bibitem[Ito14]{Ito+2014+151+174}
Atsushi Ito.
\newblock Seshadri constants via toric degenerations.
\newblock {\em Journal für die reine und angewandte Mathematik (Crelles
  Journal)}, 2014(695):151--174, 2014.

\bibitem[Laz97]{lazarsfeld1997lectures}
Robert Lazarsfeld.
\newblock Lectures on linear series.
\newblock In {\em Complex algebraic geometry}, pages 161--219. American
  Mathematical Society Providence, 1997.

\bibitem[Liu02]{liu2002algebraic}
Q.~Liu.
\newblock {\em Algebraic Geometry and Arithmetic Curves}.
\newblock Oxford graduate texts in mathematics. Oxford University Press, 2002.

\bibitem[Mar68]{Martens1968}
Henrik~H. Martens.
\newblock Varieties of special divisors on a curve. ii.
\newblock {\em Journal für die reine und angewandte Mathematik}, 233:89--100,
  1968.

\bibitem[Rei88]{0a1ebf56-1d8f-3732-ba76-902ace340afa}
Igor Reider.
\newblock Vector bundles of rank 2 and linear systems on algebraic surfaces.
\newblock {\em Annals of Mathematics}, 127(2):309--316, 1988.

\bibitem[Roc97]{Rocco1997GenerationOK}
Sandra~Di Rocco.
\newblock Generation of k-jets on toric varieties.
\newblock {\em Mathematische Zeitschrift}, 231:169--188, 1997.

\bibitem[SV20]{10.1093/imrn/rnaa137}
Geoffrey Smith and Isabel Vogt.
\newblock Low degree points on curves.
\newblock {\em International Mathematics Research Notices}, 2022(1):422--445,
  06 2020.

\end{thebibliography}
\end{document}